 \documentclass[envcountsect, envcountsame]{svmult}
  \smartqed

\usepackage{mathptmx}       
\usepackage{helvet}         
\usepackage{courier}        
\usepackage{type1cm}        
\usepackage{makeidx}         
\usepackage{graphicx}        
\usepackage{multicol}        
\usepackage[bottom]{footmisc}

\usepackage{thmtools}
\usepackage{thm-restate}
\usepackage{hyperref}
\usepackage{cleveref}
\usepackage{textcomp}

\makeindex             

\usepackage{amstext,amsfonts,amssymb,amscd,amsbsy,amsmath}
\usepackage{tikz-cd}	
\usepackage{tikz}

\newcommand{\QQ}{\mathbb{Q}}

\newcommand{\ZZ}{\mathbb{Z}}
\newcommand{\CC}{\mathbb{C}}
\newcommand{\GG}{\mathbb{G}}

\DeclareMathOperator{\GL}{GL}

\DeclareMathOperator{\OO}{O}
\DeclareMathOperator{\Sp}{Sp}

\DeclareMathOperator{\SO}{SO}
\DeclareMathOperator{\Id}{Id}
\DeclareMathOperator{\Ad}{Ad}

\newcommand{\Grp}{\mathbf{AlgGrp}}

\DeclareMathOperator{\sgn}{sgn}
\DeclareMathOperator{\Hom}{Hom}
\DeclareMathOperator{\End}{End}
\DeclareMathOperator{\Pf}{Pf}
\newcommand{\kk}{\Bbbk}
\usepackage{float}

\newcommand{\Gm}{\GG_{\mathrm{m}}}

\begin{document}

\title*{The degree of $\SO(n)$}
\author{Madeline Brandt, DJ Bruce, Taylor Brysiewicz, Robert Krone, Elina Robeva}

\institute{Madeline Brandt \at Department of Mathematics, University of California, Berkeley, 970 Evans Hall, Berkeley, CA, 94720 \email{brandtm@berkeley.edu}
\and DJ Bruce \at Department of Mathematics, University of Wisconsin, 480 Lincoln Drive, Madison, WI, 53706 \email{djbruce@math.wisc.edu} 
\and Taylor Brysiewicz \at Department of Mathematics, Texas A\&M University, 155 Ireland St, College Station, TX 77840 \email{tbrysiewicz@math.tamu.edu}
\and Robert Krone \at Department of Mathematics and Statistics, Queen's University, 48 University Avenue, Kingston, ON, K7L 3N6 \email{rk71@queensu.ca}
\and Elina Robeva \at Department of Mathematics, Massachusetts Institute of Technology, 77 Massachusetts Avenue, Cambridge, MA 02139 \email{erobeva@mit.edu}
}
%
%
\maketitle

\abstract*{We provide a closed formula for the degree of $\SO(n)$ over an algebraically closed field of characteristic zero. In addition, we describe symbolic and numerical techniques which can also be used to compute the degree of $\SO(n)$ for small values of $n$.
 As an application of our results, we give a formula for the number of critical points of a low-rank semidefinite programming optimization problem. Finally, we provide some evidence for a conjecture regarding the real locus of $\SO(n)$. }

\abstract{We provide a closed formula for the degree of $\SO(n)$ over an algebraically closed field of characteristic zero. In addition, we describe symbolic and numerical techniques which can also be used to compute the degree of $\SO(n)$ for small values of $n$. As an application of our results, we give a formula for the number of critical points of a low-rank semidefinite programming optimization problem. 
Finally, we provide some evidence for a conjecture regarding the real locus of $\SO(n)$.}

\section{Introduction}
\label{introsect}

The \emph{special orthogonal group} $\SO(n,\mathbb{R})$ is the group of automorphisms of $\mathbb{R}^n$ which preserve the standard inner product and have determinant equal to one.
The complex special orthogonal group is the complexification of the special orthogonal group and can be thought of more explicitly as the group of matrices
\[
\SO(n):=\SO(n,\CC)=\left\{M\in \text{Mat}_{n,n}(\CC) \; | \; \det M =1, \quad M^{t}M=\text{Id}\right\}.
\]
As these conditions are polynomials in the entries of such a matrix, we view $\SO(n)$ as a complex variety. 

Recall that the degree of a complex variety $X$ is the generic number of intersection points of $X$ with a linear space of complementary dimension. Problem $4$ on Grassmannians in \cite{fitness} asks for a formula for the degree of the of $\SO(n)$.
Our primary result is the following theorem, which answers this question completely.

\begin{restatable}{theorem}{maintheorem}
\label{degson}
The degree of $\SO(n)$ is given by
\begin{equation*}
\deg \SO(n)=2^{n-1}\det \left( {2n-2i-2j}\choose{n-2i} \right)_{1\leq i , j \leq \lfloor{\frac n 2}\rfloor}. 
\end{equation*}
\end{restatable}
Our proof of Theorem \ref{degson} uses a formula of Kazarnovskij \cite{kazarnovskii} (see also Theorem~\ref{Kazarnovskij}) for the degree of the image of a representation of a connected, reductive, algebraic group over an algebraically closed field. By applying this formula to the case of the standard representation of $\SO(n)$ we are able to express the degree of $\SO(n)$ in terms of its root data and other invariants.

In addition to this result, Theorem \ref{non-intersecting} provides a combinatorial interpretation of this degree in terms of non-intersecting lattice paths. In contrast to Theorem \ref{degson}, the combinatorial statement has the immediate benefit of being obviously non-negative. 

\begin{remark}
Let $\kk$ be a field of characteristic zero. We can define $\SO(n,\kk)$ using the same system of equations since they are defined over the prime field $\QQ$. For $\kk$ that is not algebraically closed, the degree of a variety can be defined in terms of the Hilbert series of its coordinate ring. Since the Hilbert series does not depend on the choice of $\kk$, the degree does not either. We choose to work over $\CC$ not only for simplicity, but also so that we may use the above definition of degree.
\end{remark}

\begin{remark}
Our methods are not specific to $\SO(n)$.  The same approach can be used to compute the degree of other algebraic groups. For example, toward the end of Section~\ref{sec:degson} we provide a similar closed formula for the degree of the symplectic group. This formula is also interpreted combinatorially in Section~\ref{NILP}.
\end{remark}

In order to verify Theorem \ref{degson}, as well as explore the structure of $\SO(n)$ in further depth, it is useful to compute this degree explicitly. We were able to do this for small $n$ using symbolic and numerical computations. A comparison of the success of these two approaches, together with our formula from Theorem \ref{degson}, is illustrated by the following table.

\begin{table}[H]\label{fig:SO}
\centering
\begin{tabular}{c | c  c  c}
\hline 
\hline
$\mathbf{n}$ & \hspace{0.2 in} \textbf{Symbolic} & \hspace{0.2 in} \textbf{Numerical} &\hspace{0.2 in}  \textbf{Formula} \\ \hline
2 & 2 & 2 & 2 \\ 
3 & 8 & 8 & 8 \\ 
4 & 40 & 40 & 40 \\
5 & 384 & 384 & 384 \\ 
6 & - & 4768 & 4768 \\ 
7 & - & 111616 & 111616 \\ 
8 & - & - & 3433600 \\ 
9 & - & - & 196968448 \\
\hline
\hline
\end{tabular}
\caption{Degree of $\SO(n)$ computed in various ways}
\end{table}

This project started in the spring of 2014, when Benjamin Recht asked the fifth author to describe the geometry of low-rank semidefinite programming (see Section~\ref{SDP}). In particular, he asked why the augmented Lagrangian algorithm for solving this problem \cite{BM} almost always recovers the correct optimum despite the existence of multiple local minima. It quickly became clear that to even compute the number of local extrema, one needs to know the degree of the orthogonal group. In Section~\ref{SDP} we find a formula for the number of critical points of low-rank semidefinite programming (see Theorem~\ref{thm:LRSDP}).

The rest of this article is organized as follows. In Section \ref{backgroundsect} we give the reader a brief introduction to algebraic groups and state Kazarnovskij's Theorem. Section \ref{sec:degson} proves Theorem \ref{degson} by applying Kazarnovskij's Theorem and simplifying the resulting expressions. After simplification, we are left with a determinant of binomial coefficients which can be interpreted combinatorially using the celebrated Gessel-Viennot lemma which we describe in Section \ref{NILP}. The relationship between the degree of $\SO(n)$ and the degree of low-rank semidefinite programming is elaborated upon in Section \ref{SDP}. Section \ref{Computational} contains descriptions of the symbolic and numerical techniques involved in the explicit computation of $\deg \SO(n)$. Finally, in Section \ref{reality} we explore questions involving the real points on $\SO(n)$.

\section{Background}
\label{backgroundsect}

In this section we provide the reader with the necessary language to understand the statement of Kazarnovskij's Theorem (see Theorem \ref{Kazarnovskij}), our main tool for determining the degree of $\SO(n)$. We invite those who already are familiar with Lie theory to skip to the statement of Theorem \ref{Kazarnovskij} and continue to Section \ref{sec:degson} for our main result. We note, that aside from applying Theorem \ref{Kazarnovskij}, no understanding of the material in this section is necessary for understanding the remainder of the proof of Theorem \ref{degson}. A more thorough treatment of the theory of algebraic groups can be found in \cite{derksen,fulton,humphreys}.

An \emph{algebraic group} $G$ is a variety equipped with a group structure such that multiplication and inversion are both regular maps on $G$. When the unipotent radical of $G$ is trivial and $G$ is over an algebraically closed field, we say that $G$ is a \emph{reductive group}. Throughout this section, we let $G$ denote a connected reductive algebraic group over an algebraically closed field $\kk$. Let $\Gm$ denote the multiplicative group of $\kk$, so as a set, $\Gm = \kk \setminus \{0\}$.  Let $T$ denote a fixed maximal torus of $G$. By \emph{maximal torus}, we mean a subgroup of $G$ isomorphic to $\Gm^r$ and which is maximal with respect to inclusion. The number $r$ is well-defined and is called the \emph{rank} of $G$.
After fixing $T$, we define the \emph{Weyl group} of $G$, denoted $W(G)$, to be the quotient of the normalizer of $T$ by its centralizer, $W(G) = N_G(T)/Z_G(T)$. Like $r$,  $W(G)$ does not depend on the choice of $T$ up to isomorphism.

\begin{example}
\label{maxTori}
We can parametrize $\SO(2,\CC)$ by $\Gm$ via the map
\[ \mathbf{R}(t) := \frac{1}{2}\begin{pmatrix}
    t+t^{-1}    & -i(t-t^{-1}) \\
    i(t-t^{-1}) &  t+t^{-1}
   \end{pmatrix}, \]
which is in fact a group isomorphism.  (Note that $\mathbf{R}(e^{i\theta})$ is the rotation matrix by angle $\theta$.)  Therefore $\SO(2)$ has rank 1.

Fix $r \in \mathbb{N}$. Then
\begin{align*}
T_{2r}&:=\left\{
\begin{pmatrix}
\mathbf{R}(t_1) & 0 & 0 & \cdots & 0\\
0 & \mathbf{R}(t_2) & 0 & \cdots & 0 \\
\vdots & \vdots & \vdots & \ddots & \vdots \\
0 & 0 & 0 & \cdots & \mathbf{R}(t_r)
\end{pmatrix} \Bigg|  t_i \in \Gm \right\}\cong \SO(2)^{r}\subset \SO(2r)\\
 T_{2r+1}&:=\left\{
\begin{pmatrix}
\mathbf{R}(t_1) & 0 & 0 & \cdots & 0 & 0\\
0 & \mathbf{R}(t_2) & 0 & \cdots & 0 & 0\\
\vdots & \vdots & \vdots & \ddots & \vdots & \vdots\\
0 & 0 & 0 & \cdots & \mathbf{R}(t_r) & 0\\
0 & 0 & 0 & \cdots & 0 & 1
\end{pmatrix}\; \Bigg| \; t_i \in \Gm \right\}\cong \SO(2)^{r}\subset \SO(2r+1)
\end{align*}
are maximal tori of rank $r$ of their respective groups. Therefore, $\text{rank}(\SO(2r))=\text{rank}(\SO(2r+1))=r$ and we see that the rank of $\SO(n)$ depends fundamentally on the parity of $n$.
\end{example}

The \emph{character group} $M(T)$ is the set of algebraic group homomorphisms from $T$ to $\Gm$, i.e. group homomorphisms defined by polynomial maps,
\begin{equation*}
M(T):=\Hom_{\Grp}(T,\Gm).
\end{equation*}
Since $T$ is isomorphic to $\Gm^r$, all such homomorphisms must be of the form
\[ (t_1,\ldots,t_r) \mapsto t_1^{a_1}\cdots t_r^{a_r} \]
with $a_1,\ldots,a_r$ integers. This character group is isomorphic to $\ZZ^r$ and for this reason it is often called the character lattice. Dual to this is the \emph{group of 1-parameter subgroups}
\begin{equation*}
N(T):=\Hom_{\Grp}(\Gm,T),
\end{equation*}
which is also isomorphic to $\ZZ^r$. Indeed, each 1-parameter subgroup is of the form $t \mapsto (t^{b_1},\ldots,t^{b_r})$ for integers $b_1,\ldots,b_r$.  There exists a natural bilinear pairing between $N(T)$ and $M(T)$, given by
\begin{align*}
M(T)\times N(T)&\rightarrow \Hom_{\Grp}(\Gm,\Gm)\cong\ZZ\\
\langle \chi, \sigma \rangle &\mapsto \chi \circ \sigma.
\end{align*}

Now if $\rho:G\rightarrow{}\GL(V)$ is a representation of $G$ we attach to it special characters called weights. A \emph{weight} of the representation $\rho$ is a character $\chi\in M(T)$ such that the set
\[
V_\chi:=\bigcap_{s\in T}\ker(\rho(s)-\chi(s)\Id_V)
\] 
is non-trivial. This condition is equivalent to saying that all of the matrices in $\{\rho(s) \; | \; s\in T\}$ have a simultaneous eigenvector $v\in V$ such that the associated eigenvalue for $\rho(s)$ is $\chi(s)$. We will use $C_V$ to denote the convex hull of the weights of the representation $\rho$.

\begin{example}\label{ex:weight}
An example that will be important for us later will be the standard representation coming from the natural embedding $\rho:\SO(n) \to \GL(\CC^n)$.  For any $t \in \Gm$, the matrix $\mathbf{R}(t) \in \SO(2)$ has eigenvectors $e_1 + ie_2$ and $e_1 - ie_2$ with eigenvalues $t$ and $t^{-1}$ respectively.  From the explicit description of $T$ in Example \ref{maxTori} we see that the eigenvectors of $\rho(t_1,\ldots,t_r)$ are all vectors of the form $e_{2j-1} \pm ie_{2j}$ with $1\leq j \leq r$ and the eigenvalues are $t_1^{\pm 1},\ldots,t_r^{\pm 1}$.  These eigenvalues, viewed as characters, are the weights of $\rho$.  Additionally when $n = 2r+1$, we have that $e_{2r+1}$ is an eigenvector with eigenvalue 1, corresponding to the trivial character.
\end{example}

Another representation of a matrix group $G \subseteq \End(V)$ is the \emph{adjoint representation}, $\Ad:G \to \GL(\End(V))$, with $\Ad(g)$ the linear map defined by $A \mapsto gAg^{-1}$.
The \emph{roots} of $G$ are the weights of the adjoint representation.  Given a linear functional $\ell$ on $M(T)$, we  define the \emph{positive roots} of $G$ with respect to $\ell$ to be the roots $\chi$ such that $\ell(\chi)>0$. We denote the positive roots of $G$ by $\alpha_1,\ldots,\alpha_l$. For the algebraic groups in this paper, we can choose $\ell$ to be the inner product with the vector $(r,r-1,\ldots,1)$ so that a root of the form $e_j-e_k$ is positive if and only if $j<k$. To each root $\alpha$, we associate a \emph{coroot} $\check \alpha$, defined to be the linear function $\check \alpha (\vec{x}):= \frac{2\langle\vec{x},\alpha\rangle}{\langle\alpha,\alpha\rangle}$ where $\langle , \rangle$ must be $W(G)$-invariant. Throughout this paper, we fix this to be the standard inner product.

\begin{example}
We now compute the roots of $\SO(n)$, starting with $n$ even.  It can be shown that the simultaneous eigenvectors of $\Ad(s)$ over all $s \in T$ are matrices $A$ with the following structure. These matrices are zero outside a $2\times 2$ block $B$ in rows $2j-1,2j$ and columns $2k-1,2k$ for some $1\leq j,k \leq r$. Furthermore, $B = v_1v_2^T$ with each $v_k$ equal to one of the eigenvectors of $\mathbf{R}(t)$, $e_1 \pm ie_2$.
Indeed, suppose $s \in T$ has blocks along the diagonal $\mathbf{R}(t_j)$ with $t_1,\ldots,t_r \in \Gm$.  Then $\Ad(s)(A)$ will also be zero except in the same $2\times 2$ block, and that block will be
 \[ \mathbf{R}(t_j)B\mathbf{R}(t_k)^T = t_j^{\pm 1}t_k^{\pm 1}B, \]
where the signs in the exponents depend on the choices of $v_1$ and $v_2$.
Thus the roots of $\SO(2r)$ are the characters of the form $t_j^{\pm 1}t_k^{\pm 1}$ for $1\leq j,k \leq r$.

In the case that $n$ is odd, $A$ has an extra row and column.  Consider $A$ with support only in the last column.  Then for $s \in T$, $\Ad(s)(A) = sAs^{-1}$ but $s^{-1}$ acts trivially on the left, while $s$ acts on the last column as an element of $\GL(\CC^n)$ as in the standard representation.  As in Example \ref{ex:weight} we get weights $t_1^{\pm 1},\ldots,t_r^{\pm 1},1$.  The same weights appear for $A$ with support in the last row.
\end{example}
 
Associated to $G$ is a Lie algebra $\mathfrak{g}$, which comes equipped with a Lie bracket $[\ ,\ ]$. A \emph{Cartan subalgebra} $\mathfrak{h}$ is a nilpotent subalgebra of $\mathfrak{g}$ that is self-normalizing, meaning if $[x,y] \in \mathfrak{h}$ for all $x \in \mathfrak{h}$, then $y \in \mathfrak{h}$.
  Let $S(\mathfrak{h}^*)$ be the ring of polynomial functions on $\mathfrak{h}$. The Weyl group $W(G)$ acts on $\mathfrak{h}$, and this extends to an action of $W(G)$ on $S(\mathfrak{h}^*)$. The space $S(\mathfrak{h}^*)^{W(G)}$ of polynomials which are invariant up to the action of $W(G)$ is generated by $r$ homogeneous polynomials whose degrees, $c_1+1,\ldots,c_r+1$, are uniquely determined. 
The values $c_1, \ldots, c_r$ are called \emph{Coxeter exponents}.

We are now prepared to state Kazarnovskij's theorem.

\begin{theorem}[Kazarnovskij's Theorem, Prop 4.7.18 \cite{derksen}]\label{Kazarnovskij}
Let $G$ be a connected reductive group of dimension $m$ and rank $r$ over an algebraically closed field. If  $\rho:G\rightarrow{}\GL(V)$ is a representation with finite kernel then,
\[
\deg\overline{\rho\left(G\right)}=\frac{m!}{|W(G)|(c_1!c_2!\cdots c_r!)^2|\ker(\rho)|}\int_{C_V}(\check{\alpha}_{1}\check{\alpha}_2\cdots\check{\alpha}_l)^2dv.
\]
where $W(G)$ is the Weyl group, $c_i$ are Coxeter exponents, $C_V$ is the convex hull of the weights, and $\check  \alpha_i$ are the  coroots.
\end{theorem}

If $\rho$ is the standard representation for an algebraic group $G$, then it follows that $\deg \overline{\rho(G)}=\deg G$. Therefore, in order to compute $\deg \SO(n)$, all we must do is apply this theorem for the standard representation of $\SO(n)$. The relevant data for this theorem is given in Table \ref{data} below for $\SO(n)$ and $\Sp(n)$.

\begin{table}[H]
\label{data}
\caption{}
\begin{tabular}{  l  c c c c c c}
\hline
\hline
Group& Dimension\hspace{0.18 in}& Rank \hspace{0.15 in}& Positive Roots\hspace{0.15 in} & Weights  \hspace{0.15 in} & $|W(G)|$\hspace{0.15 in} &  Coxeter Exponents  \\ \hline \\ 
$\SO(2r)$& ${2r}\choose{2}$ & $r$ &$\{e_i \pm e_j\}_{ i<j }$ & $\{\pm e_i\}$ & $r!2^{r-1}$ & $1,3,\ldots,2r-3,r-1$ \\ \\ 
 $\SO(2r+1)$ & ${2r+1}\choose{2}$& $r$ & $\{e_i \pm e_{j}\}_{i<j} \cup \{e_i\}$ & $\{\pm e_i\}$ & $r! 2^r$ & $1,3, 5, \ldots, 2r-1$ \\ \\ 
 $\Sp(r)$ & ${2r}\choose{2}$ & $r$ &$\{e_i \pm e_j\}_{i<j }\cup\{2e_i\}$ & $\{\pm e_i\}$ & $r! 2^r$ & $1,3, 5, \ldots, 2r-1$ \\ \\ 
\hline
\hline
\end{tabular}
\end{table}

\section{Main Result: The Degree of $\SO(n)$}\label{sec:degson}

We now prove our main result, Theorem \ref{degson}. At the end of this section we use the same method to obtain a formula for the degree of the symplectic group.

We begin by directly applying Theorem~\ref{Kazarnovskij} to $\SO(2r)$ and $\SO(2r+1)$ to obtain

\begin{align}
\label{eqn:even-int}
\deg\SO(2r)&=\frac{\displaystyle \binom{2r}{2}!}{\displaystyle r!2^{r-1}(r-1)!^2\prod_{k=1}^{r-1}(2k-1)!^2} \int_{C_V} \left(\prod_{1\leq i<j\leq r}(x_i^2-x_j^2)^2\right)dv,\\
\label{eqn:odd-int}
\deg\SO(2r +1) &= \frac{\displaystyle \binom{2r +1}{2}!}{r!2^r \displaystyle \prod_{k=1}^r(2k-1)!^2} \int_{C_V} \left(\prod_{1\leq i<j\leq r}(x_i^2-x_j^2)^2\prod_{i=1}^r(2x_i)^2\right)dv.
\end{align}
Thus, to compute the degree of $\SO(n)$ it suffices to find formulas for the integrals above. We do this by first expanding the integrand into monomials, and then integrating the result. 
We use the well-known expression for the determinant of the Vandermonde matrix,
\[ \prod_{1\leq i < j\leq r}(y_j-y_i)=\sum_{\sigma\in S_r}\left(\sgn(\sigma)\prod_{i=1}^r y_i^{\sigma(i)-1}\right). \]
Substituting $y_i = x_i^2$ and squaring the entire expression yields
\begin{equation}\label{rewrite-integrand}
 \prod_{1\leq i<j\leq r}(x_i^2-x_j^2)^2 = \sum_{\sigma,\tau \in S_r}\left(\sgn(\sigma\tau) \prod_{i=1}^r x_i^{2\sigma(i)+2\tau(i)-4}\right).
\end{equation}
Additionally, we point out that every variable in the integrand is being raised to an even power and $C_V$ is the convex hull of weights, $\{\pm e_i\}$. Because of this symmetry, the integrals over $C_V$ are $2^r$ times the same integrals over $\Delta_r$, the standard $r$-simplex. We have now reduced the computation of this integral to understanding the integral of any monomial over the standard simplex. The following proposition provides a formula for this.

\begin{proposition}[Lemma 4.23 \cite{milneANT}]\label{integral-monomial}
Let $\Delta_{r}\subset \mathbb{R}^{r}$ be the standard $r$-simplex. If $\mathbf{a}=(a_1,\ldots,a_r)\in \mathbb{Z}_{>0}^r$ then
\[
\int_{\Delta_r}\mathbf{x}^\mathbf{a}d\mathbf{x}=\int_{\Delta_r}x_1^{a_1}x_{2}^{a_2}\cdots x_{r}^{a_r}dx_1dx_2\cdots dx_r=
\frac{1}{(r+\sum a_i)!}\prod_{i=1}^r a_i!.
\]
\end{proposition}

We can now get expressions for the integrals in (\ref{eqn:even-int}) and (\ref{eqn:odd-int}) directly by applying (\ref{rewrite-integrand}) and Proposition \ref{integral-monomial}.

\begin{minipage}{\textwidth}
\begin{proposition}\label{two-integrals}
Let $I_{even}(r)$ and $I_{odd}(r)$ denote the integrals in (\ref{eqn:even-int}) and (\ref{eqn:odd-int}) respectively. Then,
\[
I_{even}(r) =\frac{r!2^r}{\binom{2r}{2}!}\det\left((2i+2j-4)!\right)_{1\leq i,j\leq r}.
\]
\[
I_{odd}(r) =\frac{r!2^{3r}}{\binom{2r+1}{2}!}\det\left((2i+2j-2)!\right)_{1\leq i,j\leq r}.
\]
\end{proposition}
\end{minipage}

\begin{proof}
As mentioned above, we can compute $I_{odd}$ by considering the integrand only over the simplex. This, along with equation \eqref{rewrite-integrand} gives us that
\begin{align*}
I_{odd}(r) &= 2^r \int_{\Delta_r} \prod_{1\leq i<j\leq r} (x_i^2-x_j^2)^2 \prod_{i=1}^r (2x_i)^2 dv\\
&= 2^r \int_{\Delta_r} \left(\sum_{\sigma,\tau \in S_r}\sgn(\sigma\tau) \prod_{i=1}^r x_i^{2\sigma(i)+2\tau(i)-4}\right) \prod_{i=1}^r(2x_i)^2 dv\\
&= 2^{3r} \sum_{\sigma,\tau \in S_r}\sgn(\sigma\tau)\int_{\Delta_r}\prod_{i=1}^r x_i^{2\sigma(i)+2\tau(i)-2} dv.
\end{align*}
As the integrand is homogeneous of degree $4\binom{r}{2}+2r$, applying Proposition \ref{integral-monomial} and simplifying yields
\begin{align*}
I_{odd}(r)
&= \frac{2^{3r}}{\left(4\binom{r}{2}+3r\right)!} \sum_{\sigma,\tau\in S_r} \sgn(\sigma\tau) \prod_{i=1}^r (2\sigma(i)+2\tau(i)-2)!,
\end{align*}
which after replacing $i$ with $\sigma^{-1}(i)$ gives us
\[ \prod_{i=1}^r(2\sigma(i)+2\tau(i)-2)! = \prod_{i=1}^r(2i+2\tau\sigma^{-1}(i)-2)!. \]
Let $\rho = \tau\sigma^{-1}$.  Over all pairs $\sigma,\tau \in S_r$, the permutation $\rho$ appears as each permutation in $S_r$ exactly $r!$ times,  and $\sgn(\sigma\tau) = \sgn(\rho)$.  Therefore, we have that
\begin{align*}
I_{odd}(r)
&= \frac{r!2^{3r}}{\left(4\binom{r}{2}+3r\right)!} \sum_{\rho\in S_r} \sgn(\rho) \prod_{i=1}^r (2i+2\rho(i)-2)!\\
&= \frac{r!2^{3r}}{\binom{2r+1}{2}!} \det\left((2i+2j-2)!\right)_{1\leq i,j\leq r}.
\end{align*}
The derivation of $I_{even}$  follows precisely the same steps.\qed
\end{proof}

Theorem \ref{degson} now follows directly from the subsequent simplification.

\begin{align*}
\deg \SO(2r+1) &= \frac{2^{2r}}{(1!3!\cdots(2r-1)!)^2} \det ((2i+2j-2)!)\\
&=\frac{2^{2r}}{(1!2!\cdots (2r-1)!)} \det\left(\frac{(2i+2j-2)!}{(2i-1)!}\right)\\
&=2^{2r}\det\left(\frac{(2i+2j-2)!}{(2i-1)!(2j-1)!}\right)\\
&=2^{2r}\det\left(\binom{2i+2j-2}{2i-1}\right)_{1\leq i,j\leq r}.
\end{align*}
Reversing the order of the rows and columns of the final matrix and reindexing produces the formula given in Theorem~\ref{degson}. Similarly, for the even case, we have

\begin{align*}
\deg \SO(2r) &= \frac{2}{(1!3!\cdots(2r-3)!(r-1)!)^2} \det ((2i+2j-4)!)\\
&= \frac{2 \cdot (2^{r-1})^2}{(1!3!\cdots(2r-3)!2\cdot 4 \cdots (2r-2))^2} \det ((2i+2j-4)!)\\
&= 2^{2r-1} \det \left(\frac{(2i+2j-4)!}{(2i-2)!(2j-2)!}\right)\\
&=2^{2r-1}\det\left(\binom{4r-2i-2j}{2r-2i}\right)_{1\leq i,j\leq r}.
\end{align*}

This finishes the proof of Theorem \ref{degson}. 

Since the orthogonal group $\OO(n)$ has two components that are isomorphic to $\SO(n)$, we immediately get a formula for the degree of $\OO(n)$. 

\begin{corollary}
\label{degOn}

The degree of $\OO(n)$ is given by
\begin{equation*}
\deg \OO(n)=2^{n}\det \left( {2n-2i-2j}\choose{n-2i} \right)_{1\leq i , j \leq \lfloor{\frac n 2}\rfloor}. 
\end{equation*}
\end{corollary}  
Furthermore, as mentioned in the introduction, there is no reason, {\it a priori}, that the steps taken in this section are particular to $\SO(n)$. We now apply these methods to find the degree of $\Sp(r)$, the group of (complex) symplectic matrices.

Recall the \emph{symplectic group} over $\CC$ is defined to be 
$$
\Sp(r):=\Sp(r,\CC) = \{ M \in \text{Mat}_{2r,2r}(\CC) \ |\ M^T \Omega M = \Omega\},
$$
where 
$$
\Omega =
\begin{pmatrix}
0 & I_r \\
-I_r & 0
\end{pmatrix}.
$$

\begin{theorem}\label{thm:symplectic}
The degree of $\Sp(r)$ is given by
$$
\deg \Sp(r) =\det\left({2i + 2j - 2 \choose 2i-1}\right)_{1\leq i,j\leq r}.
$$
\end{theorem}

For $1 \leq r \leq 5$ the values of $\deg \Sp(r)$ are $2,24,1744,769408,2063048448,\ldots$. This was verified using both numerical and symbolic techniques up to $r=3$.
\begin{proof}
This is an application of Kazarnovskij's result which is completely analogous to the computation for the special orthogonal group. The integral is the same as the one for $\SO(2r+1)$ up to factors of 2, so it is evaluated in the same way, and then the expression can be simplified
\begin{align*}
\deg \left(\text{Sp}(r)\right) &=\frac{(r(2r+1))!}{r! 2^r (1!3!\cdots(2r-1)!)^2}\int_{C_V} \left(\prod_{1\leq i<j\leq r}(x_i-x_j)^2(x_i+x_j)^2\prod_{i=1}^r x_i^2\right)dv \\
&=\frac{1}{(1!3!\cdots(2r-1)!)^2} \det\left((2i+2j-2)!\right)_{1\leq i,j\leq r}\\
&=\det\left({2i + 2j - 2 \choose 2i-1}\right)_{1\leq i,j\leq r}.
\end{align*}\qed
\end{proof}

We remark that our formula for $\deg \text{Sp}(r)$ is particularly interesting because the determinant in Theorem \ref{thm:symplectic} is the same as the determinant in Theorem \ref{degson} when $n=2r+1$.

\begin{corollary}
\label{cor:sp}
$$\deg \SO(2r+1)=2^{2r}\deg \text{Sp}(r)$$
\end{corollary}
\begin{proof}
Sending the $(i,j)$ entry of the matrix in Theorem \ref{thm:symplectic} to the ${(r-i+1,r-j+1)}$ entry does not change the determinant and gives us that $$\deg \text{Sp}(r)=\det\left({4r+2-2i-2j \choose 2r+1-2i}\right)_{1\leq i,j\leq r}.$$
When $n=2r+1$, this is the matrix appearing in Theorem \ref{degson} and all that is different is the coefficient in front. Accounting for this coefficient finishes the proof.\qed
\end{proof}

\section{Non-Intersecting Lattice Paths}
\label{NILP}

The formulas given in the previous section for the degrees of $\SO(n),\OO(n),$ and $\text{Sp}(r)$ can be interpreted as a count of non-intersecting lattice paths via the Gessel-Viennot Lemma \cite{GV}. 

\begin{lemma}[Gessel-Viennot (Weak Version)]
\label{GV}
Let $A=\{a_1, \ldots,a_r\}$, $B=\{b_1, \ldots,b_r\}$ be collections of lattice points in $\mathbb{Z}^2$. Let $M_{i,j}$ be the number of lattice paths from $a_i$ to $b_j$ using only unit steps in either the North or East direction. If the only way that a system of these lattice paths from $A \to B$ do not cross each other is by sending $a_i\mapsto b_i$, then the determinant of $M$ equals the number of such non-intersecting lattice paths.
\end{lemma}

The number of lattice paths from $(0,0)$ to $(i,j)$ is the binomial coefficient ${i+j} \choose i$. Since the matrix involved in the formulas for the degrees of $\SO(n),$ $\OO(n),$ and $\text{Sp}(r)$ has binomial coefficients as entries, it is natural to search for a interpretation of its determinant via Gessel-Viennot.

\begin{theorem}\label{non-intersecting}
Let $N(n)$ count the number of non-intersecting lattice paths from $A(n):=\{a_i\}_{i=1}^{\lfloor \frac n 2 \rfloor}$ to $B(n):=\{b_j\}_{j=1}^{\lfloor \frac n 2 \rfloor}$ where $a_i=(2i-n,0)$ and $b_j=(0,n-2j)$. Then
\begin{align*}
\deg\SO(n) &= 2^{n-1}N(n), \\
\deg\OO(n) &= 2^{n}N(n), \\
\deg\text{Sp}(r) &= N(2r+1).
\end{align*}

\end{theorem}
\begin{proof}
It is enough to prove this theorem for $\SO(n)$ and apply Corollaries \ref{degOn} and \ref{cor:sp}. Noticing that the matrix appearing in Theorem \ref{degson} is the minor of Pascal's matrix which skips every other row and every other column up to $\lfloor \frac n 2 \rfloor$ shows that we have a correct point configuration for Gessel-Viennot.\qed
\end{proof}

\begin{example}
Figure \ref{fig:GVExample} computes that $N(5)=24$ by explicitly listing all $24$ non-intersecting lattice paths from $A(5)$ to $B(5)$. Then, according to Theorem \ref{non-intersecting}, we see that $\deg \SO(5)=2^4 \cdot 24=384$, $\deg \OO(5)=2^5 \cdot 24=768$, and $\deg \text{Sp}(2)=24$.

\begin{figure}
  \centering
  \includegraphics[width=0.55 \linewidth]{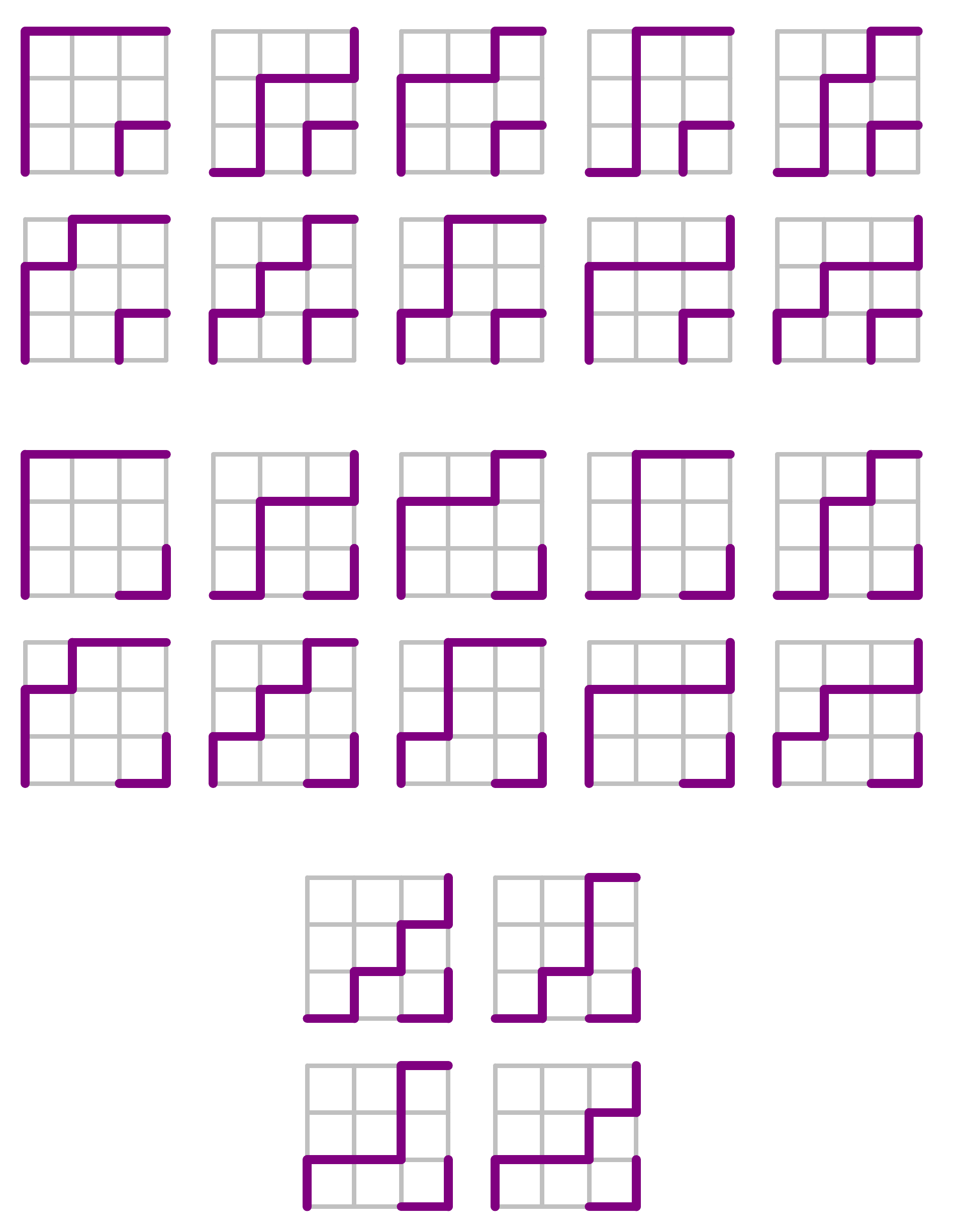}
  \caption{All $24$ instances of non-intersecting lattice paths from $A(5)$ to $B(5)$}
\label{fig:GVExample}
\end{figure}

\end{example}

Theorem \ref{non-intersecting} suggests a relationship between these non-intersecting lattice paths and the degrees of $\SO(n),\OO(n),$ and $\text{Sp}(r)$. Such a direct interpretation could be interesting, and so we pose the question:

\begin{question}
\label{ourQuestion}
Does Theorem~\ref{non-intersecting} have a deeper combinatorial interpretation?
\end{question}

Because the formula for the degree of the symplectic group has no coefficient in front of the lattice path count in Theorem \ref{non-intersecting}, studying the combinatorial meaning of the degree of $\text{Sp}(r)$ may be an ideal starting point to tackle Question \ref{ourQuestion}.

\section{An Application - The Degree of Low Rank Semidefinite Programming}
\label{SDP}
In this section we show how knowing the degree of $\SO(n)$ can be used to compute the number of critical points for a certain optimization problem (cf. Theorem \ref{thm:LRSDP}).

Consider the standard formulation of \emph{semidefinite programming}
\begin{align}\label{primalProblem}
&\text{minimize}_{X\in\mathcal S^n} \quad C\bullet X\notag\\
&\text{such that }\quad A_i\bullet X = b_i, i=1,..., m, \quad X\succeq 0.
\end{align}
Here $\mathcal S^n$ is the set of $n\times n$ real symmetric matrices, $b\in\mathbb Q^m$ is a vector, $C, A_1, ..., A_m\in\mathbb Q\mathcal S^n$ are matrices, and $\bullet$ denotes the trace inner product for matrices: $U\bullet V = \text{trace}(UV)$. 

Semidefinite programming can be solved in polynomial time in the size $n$ of the unknown matrix $X$ and in the number of constraints $m$. It is a widely used method in practice, and many NP-hard problems possess semidefinite relaxations \cite{BV2, GW2}. However, it is often the case that the size $n$ is very large, and solving \eqref{primalProblem} exactly can be computationally prohibitive. On the other hand, the rank $r$ of the optimal solution $X^*$ is often much smaller than $n$, and in those cases we can solve~\eqref{primalProblem} more rapidly by replacing $X$ by the low rank positive semidefinite matrix $RR^T$, where $R\in\mathbb R^{n\times r}$. This idea and an algorithm to solve the new problem are due to Burer and Monteiro~\cite{BM}. The problem becomes
\begin{align}\label{newProblem}
&\text{minimize}_{R\in\mathbb R^{n\times r}} \quad C\bullet (RR^T)\notag\\
&\text{such that }\hspace{0.33cm} A_i\bullet (RR^T) = b_i, ~~i=1, ..., m.
\end{align}
The constraint $X\succeq 0$ is now implicit and the number of variables has decreased from $n^2$ to $nr$. However, the objective function and the constraints are no longer linear; instead, they are quadratic and the feasible set is non-convex. In \cite{BM} ~Burer and Monteiro propose a fast algorithm for solving \eqref{newProblem}. Despite the existence of multiple local minima, in practice this algorithm quickly finds the global minimum. It starts by choosing the rank $r=1$, and increments it until $C - \sum_{i=1}^my_i A_i\succeq 0$, which ensures that we have arrived at the smallest optimal $r$.  For each fixed rank $r$, the optimization problem \eqref{newProblem} is non-convex, and its appealing behavior still remains to be examined. In Theorem~\ref{thm:LRSDP} we give a formula for the number of critical points of this optimization problem. 
We call a {\em critical point} of the optimization problem \eqref{newProblem} any point $(R, y)$ which satisfies the Lagrange multipliers equations arising from this problem. Here $y$ is a vector of size $m$, and its entries $y_1,\dots, y_m$ are the new dual variables introduced for the $m$ constraints in \eqref{newProblem} (see equation \eqref{Lagrangian}).
Before we state our theorem, we need the following definition.
\begin{definition} Let
\vspace{-0.2cm}
$$\psi_i = 2^{i-1}, \quad \psi_{i,j} = \sum_{k=i}^{j-1}\binom{i+j-2}k \text{ when } i < j,$$
\vspace{-0.1cm}
and
$$\psi_{i_1,\dots, i_r}= \Pf(\psi_{i_k, i_l})_{1\leq k < l\leq r} \text{ if } r \text{ is even},$$
$$\psi_{i_1,\dots, i_r}= \Pf(\psi_{i_k, i_l})_{0\leq k < l\leq r} \text{ if } r \text{ is odd}$$
where $r>2$, $\psi_{0, k } = \psi_{k}$, and $\Pf$ denotes the Pfaffian. Then, define $\delta(m, n, r)$ as
$$\delta(m, n, r) = \sum_{I} \psi_I\psi_{I^c},$$
where the sum runs over all strictly increasing subsequences $I = \{i_1, ..., i_{n-r}\}$ of $\{1, ..., n\}$
of length $n - r$ and such that $i_1 + ... + i_{n-r} = m$.
\end{definition}

\begin{theorem}\label{thm:LRSDP}
The number of critical points of the low-rank semidefinite programming algorithm \eqref{newProblem} is
$$2 (\deg\SO(r)) \delta(m, n, r).$$
\end{theorem}

\begin{remark} The number $\delta(m, n, r)$ is called the {\em algebraic degree of semidefinite programming} and was originally defined in \cite{NRS} as the number of critical points of the original semidefinite programming problem \eqref{primalProblem} for which the matrix $X$ has rank~$r$. The final formula for it was computed in \cite{VR}.
\end{remark}

\noindent\textit{Proof of Theorem~\ref{thm:LRSDP}:} In order to analyze the optimality conditions for the program (\ref{newProblem}) for a fixed $r$, consider the Lagrangian function
\begin{align}\label{Lagrangian}
L(R, y) = C\bullet (RR^T) - \sum_{i=1}^my_i(A_i\bullet(RR^T) - b_i).
\end{align}
Taking derivatives, we find out that the critical points $(R, y)$ of this optimization problem are given by the Lagrange multipliers equations:
\begin{align}\label{equations}
\left(C - \sum_{i=1}^my_iA_i\right)RR^T&= 0\\
A_i\bullet(RR^T) &= b_i, i=1,2,...,m.\notag
\end{align}
In addition, those critical points relevant for applications have to be real and have to satisfy
\begin{align}\label{addEquation}
\left(C - \sum_{i=1}^my_iA_i\right)\succeq 0,
\end{align}
since this is the constraint in the dual to the optimization problem \eqref{newProblem}. However, in this article we are primarily concerned with counting all of the critical points.
Analogously, in \cite{NRS} Nie, Ranestad, and Sturmfels show that the critical points of the original semidefinite programming problem \eqref{primalProblem} satisfy
\begin{align}
\left(C - \sum_{i=1}^my_iA_i\right)X &= 0,\label{optimalSolutions1}\\
A_i\bullet X &= b_i, i=1,...,m.\label{optimalSolutions2}
\end{align}
In addition, the critical points relevant for applications have to satisfy
\begin{align}
\left(C - \sum_{i=1}^my_iA_i\right)\succeq 0 \text{ and } X&\succeq 0,\label{optimalSolutions3}
\end{align}
but these conditions are disregarded and the total number of critical points is counted.
Nie, Ranestad, and Sturmfels show that the number of solutions $(X, y)$ to \eqref{optimalSolutions1}-\eqref{optimalSolutions2},  for which the rank of $X$ is $r$, equals $\delta(m, n, r)$ (c.f. Definition 1). Comparing our system of equations \eqref{equations} to the equations \eqref{optimalSolutions1}-\eqref{optimalSolutions2}, we see that the fiber of the map $(R, y) \mapsto (RR^T, y)$ above each point $(X, y)$, satisfying \eqref{optimalSolutions1}-\eqref{optimalSolutions2}, consists of all points $(R, y')$, satisfying \eqref{equations}, and such that $y' = y$ and $X = RR^T$. Given $X$ and one matrix $R$ such that $X = RR^T$, all other matrices $S$ such that $(S, y)$ is in the fiber above $(X, y)$ have the form $S = RU$ where $U$ runs over all orthogonal $r\times r$ matrices.  In other words, this fiber is isomorphic to a copy of the orthogonal group $O(r)$. Therefore, the number of solutions to \eqref{equations} is equal to $2(\deg\SO(r))\delta(m, n, r)$.
\qed

The number of critical points of low-rank semidefinite programming grows rapidly with the rank $r$, and the appealing behavior of the augmented Lagrangian algorithm \cite{BM} still needs to be explained. It would be quite interesting and relevant for applications to examine how many of the critical points computed in Theorem 5 and in \cite{NRS} are real, and moreover, how many of them satisfy the additional linear matrix inequality constraints \eqref{addEquation} and \eqref{optimalSolutions3} respectively. This is a real algebraic problem and would involve counting polynomial system solutions over semialgebraic sets. This question is addressed more in Section \ref{reality}.

\section{Computational Methods}
\label{Computational}
Although we have already derived a formula for the degree of $\SO(n)$, it is natural to want to compute this degree explicitly for particular values of $n$. Aside from merely verifying the formula in Theorem \ref{degson}, the computation of this degree gives us access to other useful data along the way. In our case, this manifests itself as either a Gr\"{o}bner basis or a witness set for $\SO(n)$. Once computed, either may be used in further computations (such as those done in Section \ref{reality} using witness sets). Additionally, $\SO(n)$  serves as a prime example of when numerical algorithms are better suited for computation than other techniques. Even though our computations focus on $\SO(n)$, these methods are useful for studying many other varieties. 

In this section, we describe three techniques which can compute the degree of a variety: a Gr\"obner basis algorithm, polynomial homotopy continuation, and a numerical monodromy algorithm. The first is symbolic and the last two use numerical algebraic geometry. The results of our symbolic and numerical computations for $\deg \SO(n)$ appear in the first two columns of Table \ref{fig:SO}. Code for each method is given in the appendix.

Using Gr\"obner bases, we were able to compute the degree of $\SO(n)$ for $n \leq 5$.  The standard algorithm computes a Gr\"{o}bner basis for the ideal of $\SO(n)$ over $\mathbb{Q}$ and uses this to produce the Hilbert polynomial. However, since the dimension of $\SO(n)$ grows quadratically in $n$, this method quickly becomes computationally infeasible. Computing a Gr\"{o}bner basis over a finite field can speed up the computation, but this method is still quite slow.
 
A common numerical algorithm for computing the degree of $\SO(n)$ over $\mathbb{C}$ randomly chooses an affine linear space $\mathcal L$ of complementary dimension and counts the number of complex solutions $S$ to the zero-dimensional system corresponding to $\SO(n) \cap \mathcal L$. This data is contained in the triple $(\SO(n),\mathcal L, S)$ which is called a \emph{witness set} for $\SO(n)$. This is the fundamental data type in numerical algebraic geometry in the sense that the computation of a witness set is often a necessary step for other numerical algorithms. Such techniques include sampling points on the variety at a rapid rate, studying its asymptotic behaviour, computing its monodromy group, or even studying its real locus, as we do in Section \ref{reality}. Both numerical algorithms presented below produce a witness set for $\SO(n)$.

\emph{Polynomial homotopy continuation} computes a witness set by solving a system of polynomials describing these points. Briefly, this method begins with a ``start'' polynomial system that has similar structure to the ``target'' system we want to solve, but for which the solutions are obvious. The solutions of the start system are quickly tracked through a homotopy towards those of the target system \cite{SW}. The most basic start system one uses for this technique has a solution count equal to the product of the degrees of the polynomials in the target system. This number is called the \emph{B\'ezout bound} and for our case is equal to $2^{n(n+1)/2}$ (for $n=6$, this is already $2097152$). The \emph{polyhedral start system}, however, has a solution count equal to the mixed volume of the Newton polytopes of these polynomials. In our case, this count provides no savings as it is equal to the B\'ezout bound. Because of how many paths need to be tracked with this method, we were only able to compute the degree of $\SO(n)$ up to $n=5$ with this method, just like with Gr\"obner bases.

The method that proved to be the most efficient takes advantage of the monodromy group of $\SO(n)$. The basic idea is that if we know some point on a linear cut $W=\mathcal{L} \cap \SO(n)$, we can track this solution from the slice $W$ along some path $\gamma$ to another slice $W'$ using homotopy methods. Tracking this solution along a different path $\gamma'$ back to $W$ then induces a permutation $\sigma_{\gamma,\gamma'}$ on the points in $W$. Therefore, applying this action to a point $x_0 \in W$ will likely produce a new point $\sigma_{\gamma,\gamma'}(x_0) \in W$. One iterates this process hoping to populate the witness set associated to $W$. Other than knowing the degree {\it a priori}, stopping criteria for this method tend to be heuristic in nature: one can wait until the algorithm fails to produce new points (suggesting there are no new points to be found) or one can compute a \emph{trace test} \cite{tracetest} which numerically decides whether or not a witness set is complete. This monodromy method has been implemented in the package \emph{ monodromySolver} for \emph{Macaulay2} \cite{M2} and is explained in much more detail in \cite{duff2016solving}.

\begin{remark}
A major computational result arising from this project was the computation of witness sets for $\SO(6)$ and $\SO(7)$. This was done in $630$ and $42790$ seconds respectively using \emph{ monodromySolver}. The algorithm stopped when no new points were found on ten consecutive iterations.
\end{remark}

\section{Real Points on $\SO(n)$}
\label{reality}

An interesting question pertaining to $\SO(n)$ is whether or not this variety always admits some witness set consisting of only real points. Since tracking points of one witness set to those of another is computationally inexpensive via homotopy continuation, we use this method to generate experimental data regarding real points on witness sets of $\SO(3),$ $\SO(4),$ and $\SO(5)$.  

The number of coefficients needed to produce a linear cut of $\SO(n)$ is $(n^2+1) {{n}\choose 2}$. We randomly choose these coefficients using the \emph{random} function in \emph{Macaulay2} in order to sample linear cuts of $\SO(n)$. We then use homotopy continuation to track solutions of a precomputed witness set to those lying on the randomly chosen linear cut. Finally, we determine how many solutions in the new cut are real by checking whether each solution is within a $0.001$ numerical tolerance of a real point coordinate-wise. One can certify the results using the software \emph{alphaCertify} which implements Smale's $\alpha$ theory \cite{alphaCertified}. For the sake of speed, we chose not to certify all of the results, but instead certify at least one witness set achieving the observed maximum of real points (cf. Table \ref{fig:SO3Real}, Table \ref{fig:SO4Real}, and Table \ref{fig:SO5Real}).

After computing $1398000$, $1004100$, and $48200$ witness sets for $\SO(3),\SO(4),$ and $\SO(5)$ respectively, we have summarized the number of real solutions found in each witness set in the frequency tables and histograms below. Explicit data and code used can be found in \cite{dataSite}. Note that very rarely, numerical failures occur because the path that homotopy continuation is being performed over is ill-conditioned (for example, almost singular). These occurrences are also tallied below under ``fail''.

\begin{table}[H]
\centering
\begin{tabular}{c |c| c| c | c | c |c||c}
\hline 
\hline
\#(Real Solutions)&\bf{Fail}&\bf{0} & \bf{2} &  \bf{4} & \bf{6}&  \bf{8}&\bf{Total} \\ \hline
Frequency &2& 285676 & 420049 & 549875 & 127699 & 14699&1398000 \\ 
\hline
\hline
\end{tabular}
\caption{Number of real points on witness sets of $\SO(3)$}
\label{fig:SO3Real}
\end{table}
\begin{table}[H]
\centering
\begin{tabular}{  c | c | c | c | c | c | c | c }
\hline 
\hline
\#(Real Solutions) &{\bf Fail}& \bf 0 & \bf 2& \bf 4& \bf 6& \bf {8}& \bf {10}\\ \hline
Frequency 
&51
&183427 
&108273
&132143
&156010
&159630
&124843\\
\hline
\hline
\end{tabular}
\begin{tabular}{ |c|c|c|c|c| c | c | c | c | c | c | c | c ||c}
\hline 
\hline
 \bf {12}& \bf {14}& \bf {16}& \bf {18}& \bf {20} &\bf {22} &\bf{24}&\bf{26}&\bf{28}&\bf{30}&\bf{32}&$\cdots$&\bf{40}&\bf{Total}\\ \hline

76965
&38243
&16150
&5780
&1897
&510
&145
&23
&9
&1
&0
&$\cdots$
&0
&1004100\\ 
\hline
\hline
\end{tabular}
\caption{Number of real points on witness sets of $\SO(4)$}
\label{fig:SO4Real}
\end{table}
\begin{table}[H]
\centering
\begin{tabular}{c|c|c|c | c | c | c | c | c | c | c     }
\hline 
\hline
\#(Real Solutions) &{\bf Fail}& \bf 0 & \bf 2& \bf 4& \bf 6& \bf {8}& \bf {10}& \bf {12}& \bf {14}& \bf {16}\\ \hline
Frequency
&81
&6162
&2628
&2377
&2306
&2275
&2272
&2275
&2383
&2473\\ 
\hline
\hline
\end{tabular}
\begin{tabular}{|c|c|c|c|c|c|c|c|c|c|c|c|c|c}
\hline
\hline
 \bf {18}&  \bf {20}&  \bf {22}&  \bf {24}&  \bf {26}&  \bf {28}&  \bf {30}&  \bf {32}&  \bf {34}&  \bf {36}&  \bf {38}&  \bf {40}&  \bf {42}& \bf {44}  \\ \hline

2497
&2527
&2504
&2485
&2280
&2009
&1755
&1644
&1331
&1051
&802
&591
&468
&362
\\
 \hline
 \hline
 \end{tabular}
\begin{tabular}{  | c | c | c | c | c | c | c | c | c | c | c | c | c | c|c|c|c|c|c||c}
\hline 
\hline
  \bf {46}& \bf {48}& \bf {50}& \bf {52}& \bf {54}& \bf {56}& \bf {58}& \bf {60}& \bf {62} &\bf{64} &\bf {66}&\bf {68}&\bf {70}&\bf {72}&\bf {74}&\bf {76}&\bf {78}&$\cdots$&\bf{384}&\bf{Total}\\ \hline
235
&150
&118
&60
&44
&21
&16
&8
&4
&3
&1
&0
&0
&0
&0
&2
&0
&$\cdots$
&0
&48200\\ 
\hline
\hline
\end{tabular}
\caption{Number of real points on witness sets of $\SO(5)$}
\label{fig:SO5Real}
\end{table}

\begin{figure}\label{fig:so4histogram}
  \centering
  \includegraphics[width=0.7 \linewidth]{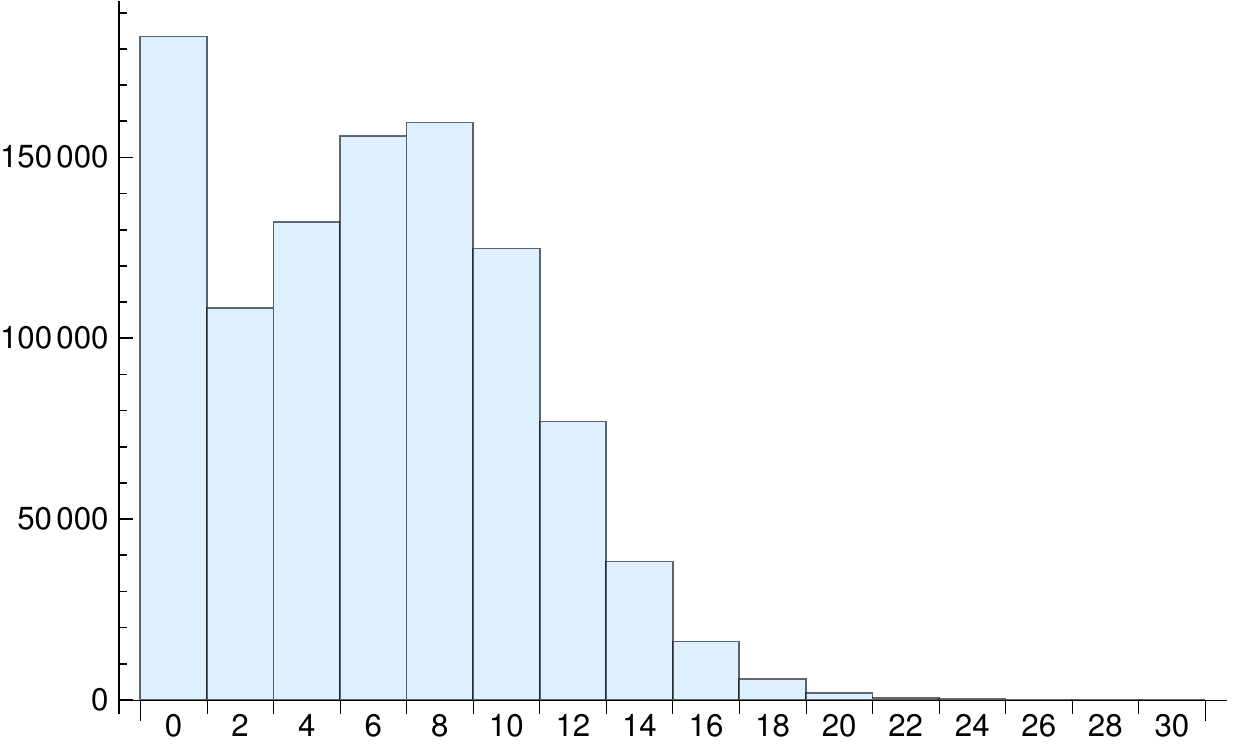}
  \caption{Histogram counting real points on witness sets of $\SO(4)$}
\end{figure}
\begin{figure}\label{fig:so5histogram}
  \centering
  \includegraphics[width=0.7 \linewidth]{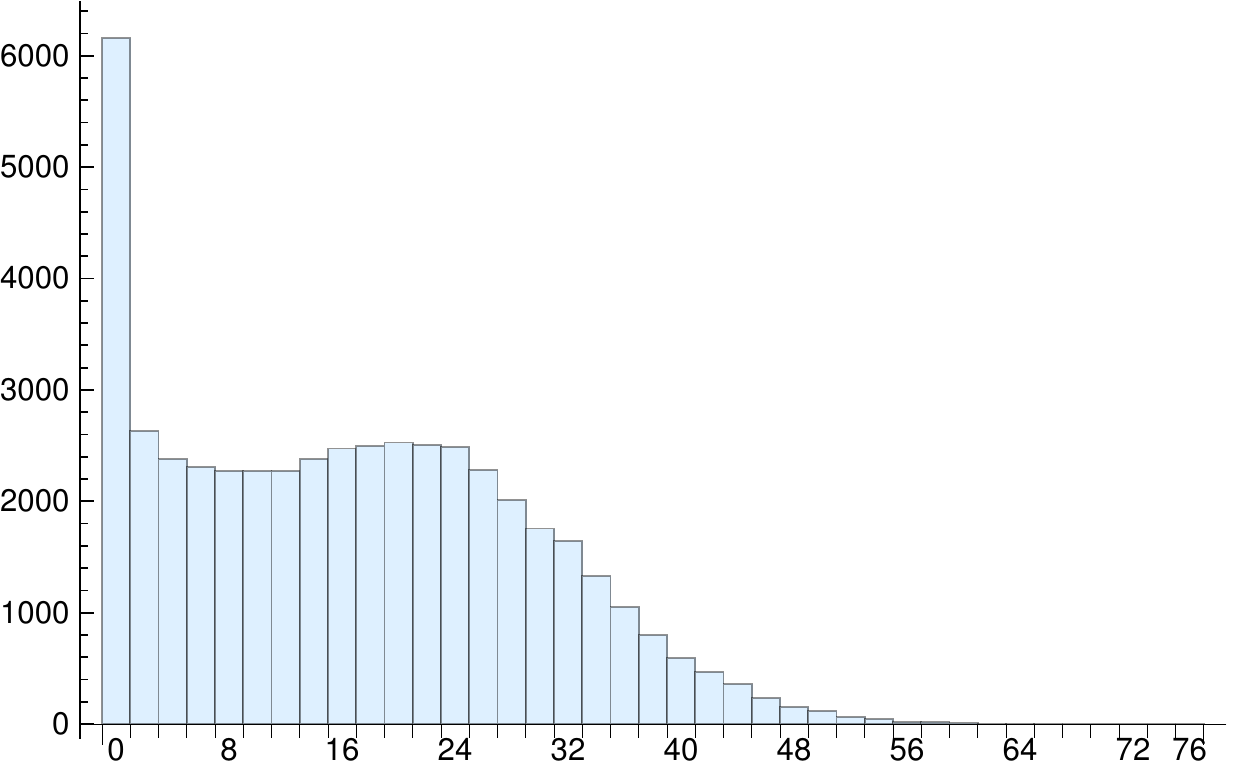}
  \caption{Histogram counting real points on witness sets of $\SO(5)$}
\end{figure}

In each case, we were able to find a witness set which failed to have any real solutions on it. This is unsurprising as $\SO(n)$ is compact over the real numbers. Despite the fact that all witness sets computed for $\SO(4)$ had fewer than $40$ solutions, and for $\SO(5)$, fewer than $384$, there is little evidence suggesting that a non-trivial upper bound for the number of real solutions on a witness set of $\SO(n)$ exists. 
We end with a conjecture.  
\begin{conjecture}
For any $n$, $\SO(n)$ admits some real witness set.
\end{conjecture}

\section*{Appendix: Macaulay2 Code}
This section contains code which computes the degree of $\SO(n)$ for various $n$ using Gr\"obner bases, polynomial homotopy continuation, \emph{MonodromySolver}, and Theorem~\ref{degson} respectively. They are all done in \emph{Macaulay2}. 

First, we compute the degree of $\SO(5)$  using Gr\"{o}bner bases. The computation is done over the finite field $\mathbb{Z}_{101}$ for $\OO(5)$ and the result is halved to give the degree of $\SO(5)$.
\begin{verbatim}
n=5
R = ZZ/101[x_(1,1)..x_(n,n)]
M = genericMatrix(R,n,n)
J = minors(1,M*transpose(M)-id_(R^n))
degOn = degree J
degSOn = degOn//2
\end{verbatim}
Computing the degree of $\OO(n)$, rather than $\SO(n)$ directly, is useful because it throws out the polynomial of highest degree in the system. This is especially useful in numerical methods since they perform best with polynomials of low degree.

The code below uses the package \emph{NumericalAlgebraicGeometry} to solve the zero dimensional system given by a linear slice of $\OO(3)$. The method \emph{solveSystem} employs the standard method of polynomial homotopy continuation. Again, the answer is halved to give $\deg \SO(3)$.
\begin{verbatim}
loadPackage "NumericalAlgebraicGeometry"
n = 3
L = toList apply(
      (0,0)..(n-1,n-1), (i,j)->"x"|toString i|toString j
      )
R = CC[L]
M = genericMatrix(R,n,n)
B = M*(transpose M) - id_(R^n)
polys = flatten for i from 0 to n-1 list(
          for j from i to n-1 list B_(i,j)
        )
linearSlice = apply(
                binomial(n,2), i->random(1,R)-random(CC)
              )
S = solveSystem(polys|linearSlice);
degOn = #S
degSOn = degOn//2
\end{verbatim}

Next, we provide code that computes the degree of $\SO(7)$ using the package \emph{MonodromySolver}. We again do not include the determinant condition, but this time we do \emph{not} need to halve the result. This is because our starting point, the identity matrix, lies on $\SO(7)$ and this method only discovers points on the irreducible component corresponding to our starting point. The linear slices are parametrized by the $t$ and $c$ variables which are varied within the function \emph{monodromySolve} to create monodromy loops. The method stops when ten consecutive loops provide no new points. Although it is possible that this stopping criterion is satisfied prematurely, in our case the program stopped at the correct number, serving as a testament to the practicality of the software and also this stopping criterion.

\begin{verbatim}
loadPackage "MonodromySolver"
N=7
d=binomial(N,2)
R=CC[c_1..c_d,t_(1,1,1)..t_(d,N,N)][x_(1,1)..x_(N,N)]
M=genericMatrix(R,N,N)
B=M*transpose(M)-id_(R^N)
polys=flatten for j from 0 to N-1 list(
       for k from j to N-1 list B_(j,k)
      );
linearSlice=for i from 1 to d list(
              c_i+sum(
                flatten for j from 1 to N list(
                  for k from 1 to N list t_(i,j,k)*x_(j,k)
                )
              )
            );
G = polySystem join (polys,linearSlice)
x0coords = flatten entries id_(CC^N)
setRandomSeed 0
(p0, x0) := createSeedPair(G,x0coords)
elapsedTime (V,npaths) = 
	monodromySolve(G,p0,{x0},NumberOfNodes=>2,NumberOfEdges=>4);
--node1: 111616                                                                                                                
--node2: 111616                                                                                                                
-- 42790.9 seconds elapsed    
\end{verbatim}

Finally, for the mathematician wanting to compute the degree of $\SO(n)$ quickest, we give code that evaluates the formula in Theorem \ref{degson}.
\begin{verbatim}
degSO = method()
degSO(ZZ) := N ->(
    n := N//2;
    M := matrix for i from 1 to n list (
        for j from 1 to n list (
            binomial(2*N-2*i-2*j,N-2*i)
            )
        );
    2^(N-1)*(det M)
    )
\end{verbatim}

\begin{acknowledgement}
This article was initiated during the Apprenticeship Weeks (22 August-2 September 2016), led by Bernd Sturmfels, as part of the Combinatorial Algebraic Geometry Semester at the Fields Institute.

The authors are very grateful to Jan Draisma for his tremendous help with understanding Kazarnovskij's formula, and to Kristian Ranestad for many helpful discussions. The authors thank Anton Leykin for performing the computation of $\SO(7)$. 
The first three authors would also like to thank the Max Planck Institute for Mathematics in the Sciences in Leipzig, Germany for their hospitality where some of this article was completed. The motivation for computing the degree of the orthogonal group came from project that started by the fifth author at the suggestion of Benjamin Recht.

The first author was supported by the National Science Foundation Graduate Research Fellowship under Grant No. DGE 1106400. The second author was partially supported by the NSF GRFP under Grant No. DGE-1256259 and the Graduate School and the Office of the Vice Chancellor for Research and Graduate Education at the University of Wisconsin-Madison with funding from the Wisconsin Alumni Research Foundation.
\end{acknowledgement}

\bibliographystyle{amsalpha}
\bibliography{ref}
\end{document}